\documentclass[10pt]{amsart}

\date{\today}

\usepackage{graphicx,ctable,booktabs,enumerate, color, xcolor}
\usepackage{amssymb,amsmath,amsthm} 
\usepackage{tikz, tikz-cd}

\usepackage{microtype}

\usetikzlibrary{arrows.meta}
\usetikzlibrary{bending}
\usetikzlibrary{fadings}

\usepackage{hyperref}

\usepackage{mathrsfs}

\newcounter{prcounter}

\newcommand{\lf}{\left}
\newcommand{\ri}{\right}

\newcommand{\f}{\frac} 
 
\newcommand{\into}{\hookrightarrow}
\newcommand{\onto}{\twoheadrightarrow}

\newcommand{\wh}{\widehat}
\DeclareMathOperator{\sgn}{sgn}

\DeclareMathOperator{\rank}{rank}

\DeclareMathOperator{\Out}{Out}

\DeclareMathOperator{\Res}{Res}
\DeclareMathOperator{\Ind}{Ind}

\DeclareMathOperator{\vol}{vol}

\DeclareMathOperator{\tr}{tr}

\newcommand{\Ld}[1]{{}^L\!{#1}}

\newcommand{\m}[1]{\mathbf{#1}}
\newcommand{\mf}[1]{\mathfrak{#1}}
\newcommand{\mc}[1]{\mathcal{#1}}
\newcommand{\td}[1]{\tilde{#1}}

\newcommand{\C}{\mathbb C}
\newcommand{\R}{\mathbb R}
\newcommand{\Q}{\mathbb Q}

\newcommand{\Z}{\mathbb Z}

\newcommand{\A}{\mathbb A}

\newcommand{\SL}{\mathrm{SL}}

\newcommand{\SO}{\mathrm{SO}}

\newcommand{\GL}{\mathrm{GL}}

\newcommand{\PGL}{\mathrm{PGL}}

\newcommand{\SU}{\mathrm{SU}}

\newcommand{\eps}{\epsilon}
\newcommand{\om}{\omega}
\newcommand{\lb}{\lambda}
\newcommand{\Om}{\Omega}

\newcommand{\bs}{\backslash}
\newcommand{\1}{\m 1}

\newcommand{\disc}{\mathrm{disc}}

\newcommand{\spec}{\mathrm{spec}}
\newcommand{\geom}{\mathrm{geom}}
\newcommand{\el}{\mathrm{ell}}

\newcommand{\ur}{\mathrm{ur}}
\newcommand{\ssm}{\mathrm{ss}}

\newcommand{\inv}{\mathrm{inv}}
\DeclareMathOperator{\st}{\mathrm{st}}

\newcommand{\cusp}{\mathrm{cusp}}
\newcommand{\Mot}{\mathrm{Mot}}

\newtheorem{thm}{Theorem}[subsection]
\newtheorem{prop}[thm]{Proposition}

\newtheorem{cor}[thm]{Corollary}
\newtheorem{lem}[thm]{Lemma}

\theoremstyle{remark}
\newtheorem*{note}{Note}

\theoremstyle{definition}
\newtheorem*{dfn}{Definition}

\title[Discrete, Quaternionic Automorphic Reps. on $G_2$]{Counting Discrete, Level-$1$, Quaternionic Automorphic Representations on $G_2$}
\author{Rahul Dalal}

\begin{document}

\begin{abstract}
Quaternionic automorphic representations are one attempt to generalize to other groups the special place holomorphic modular forms have among automorphic representations of $\GL_2$. Here, we use ``hyperendoscopy'' techniques to develop a general trace formula and understand them on an arbitrary group. Then we specialize this general formula to study quaternionic automorphic representations on the exceptional group $G_2$, eventually getting an analog of the Eichler-Selberg trace formula for classical modular forms. We finally use this together with some  techniques of  Chenevier, Renard, and Ta\"ibi to compute dimensions of spaces of level-$1$ quaternionic representations. On the way, we prove a Jacquet-Langlands-style result describing them in terms of classical modular forms and automorphic representations on the compact-at-infinity form $G_2^c$. 

The main technical difficulty is that the quaternionic discrete series that quaternionic automorphic representations are defined in terms of do not satisfy a condition of being ``regular''. A real representation theory argument shows that regularity miraculously does not matter for specifically the case of quaternionic discrete series. 

We hope that the techniques and shortcuts highlighted in this project are of interest in other computations about discrete-at-infinity automorphic representations on arbitrary reductive groups instead of just classical ones.
\end{abstract}

\maketitle

\tableofcontents

\section{Introduction}
\subsection{Context}
This work first develops an ``explicit'' trace formula \eqref{mainformula} to study so-called quaternionic automorphic representations in general and then specializes it to describe level-$1$, discrete, quaternionic automorphic representations on $G_2$. Let $\mc Q_1(k)$ be the set of such representations of weight $k$ counted with multiplicity. For each $k>2$, we give a formula, \eqref{finalformula}, for $|\mc Q_1(k)|$ in terms of counts of automorphic representations on the compact-at-infinity inner form $G_2^c$ that were calculated by Chenevier and Renard in \cite{CR15} (In case the use of counts on $G_2^c$ is bothersome, section \ref{explicit} also gives relatively short closed-form formula, though this is less conceptually enlightening). We also give a Jacquet-Langlands-style result (corollary \ref{JLG2}) describing all elements of $\mc Q_1(k)$ in terms of certain automorphic representations on $G_2^c$ and certain pairs of classical modular forms.

Quaternionic automorphic representations are one way to generalize to other groups the special place holomorphic modular forms have among automorphic representations of $\GL_2$. Just like holomorphic modular forms, they are characterized by their infinite component being in a particular nice class of discrete series representations: the quaternionic discrete series of \cite{GW96}. Just like modular forms, they also have many unexpected applications and connections to other areas of mathematics. For example, they have a nice theory of Fourier expansions with interesting arithmetic content---this was described for $G_2$ in \cite{GGS02} and generalized to all exceptional groups in \cite{Pol20}. They also somehow appear in certain string theory computations involving black holes---see the end of chapter 15 in \cite{FGK18} for example. Quaternionic forms have been studied a lot by Pollack: see \cite{Pol21} for an introductory article on them and \cite{Pol18} for good exposition specifically on $G_2$-quaternionic forms. 

We study discrete, quaternionic representations in general using the trace formula: Arthur's invariant trace formula, as in \cite{Art89}, lets us analyze automorphic representations with infinite component contained in a fixed discrete series $L$-packet. However, quaternionic discrete series appear in $L$-packets with non-quaternionic members and therefore cannot be isolated within the packet without further techniques. A previous work, \cite{Dal22}, uses the stabilization of the trace formula to abstractly isolate members of the $L$-packet and prove general asymptotic bounds. Here, we demonstrate that the same techniques suffice for computing more explicit information. 

As one technical point of interest, there is a particular miracle about quaternionic discrete series that crucially underpins the result. A priori, quaternionic discrete series are not regular, implying that there may not be a compact test function at infinity whose trace picks out exactly a quaternionic discrete series without also picking up some unwanted contributions from non-tempered representations. This would preclude the use of easy trace formula arguments. However, it turns out that specifically quaternionic discrete series don't get entangled in this way, even though other members of their $L$-packet do. The proof is a computation in real representation theory. 

Applying these generalities to $G_2$ lets us develop an ``Eicher-Selberg''-style trace formula for quaternionic discrete series there. The counts of level-$1$ forms and Jacquet-Langlands-style results come from computations with it. The level-$1$ computation in particular also relies heavily on powerful shortcuts developed in \cite{CR15} and \cite{Tai17} to get exact counts of level-$1$ automorphic representations on classical groups. 

This work can be be compared to the much more difficult aforementioned papers of Chenevier, Renard, and Ta\"ibi counting level-$1$ representations with arbitrary discrete series at infinity on classical groups. Those avoided needing our real representation theory miracle by using the extremeley powerful endoscopic classification of \cite{Art13} which, in essence, gave a finer decomposition of the trace formula than pure stabilization. In particular, it allowed the isolation of summands in the trace formula that did not contain any contributions from automorphic representations with non-tempered local components through a far more complicated inductive procedure. Unfortunately, there is no endoscopic classification currently available for $G_2$.

The methods here should be able to also compute averages of Hecke eigenvalues. We also hope that the computation highlights enough general methods and shortcuts to be helpful for people interested in doing explicit computations with discrete-at-infinity automorphic representations whenever an endoscopic classification might either be unavailable or be too complicated to use. See the end of section \ref{geomfinal} for some comments on this. In particular, a very similar method, albeit with more complicated computation at infinity, should be able to ``quickly'' count the quaternionic forms on type $D_4$ studied by Martin Weissman in \cite{Wei06}. 

\subsection{Summary}
We start with the case of general groups with quaternionic discrete series at infinity. Section \ref{qdsdef} discusses quaternionic discrete series and their properties, culminating in proposition \ref{qpstrace} showing that they satisfy a property of being ``trace-distinguishable''. This is used in theorem \ref{qspec} to show that the spectral side of Arthur's invariant trace formula can be made equal to the trace of any desired finite-place test function against the space of all quaternionic representations of a fixed weight. 

The second part of section \ref{geomG2}  uses ``hyperendoscopy'' from \cite{Fer07} to simplify the geometric side and develops an expression \eqref{mainformula} for this trace that is explicit up to computing endoscopic transfers and orbital integrals. We conclude with some remarks on how to do the necessary computations and outline how they simplify in the special case of level-$1$ on $G_2$. This special case of $G_2$ then takes up the rest of the paper. 

Specifically, after some set-up work in section \ref{g2setup}, we work out what formula \eqref{mainformula} reduces to in the unramified case for $G_2$ in section \ref{gc} using a computation of the endoscopy of $G_2$. Instead of using formula \eqref{mainformula} directly, we compare it applied to $G_2$ to it applied to the compact real form $G_2^c$ to construct a formula for $I_\spec^{G_2}$ involving just $I_\spec^{G_2^c}$- and $I_\spec^H$-terms. Here, $H$ is the endoscopic group $\SL_2 \times \SL_2/\pm 1$ of $G_2$. 

Section \ref{Rtrans} then tells us which exact $I_\spec^{G_2^c}$- and $I_\spec^H$-terms appear by computing endoscopic transfers at infinity. The difficult part of this computation is pinning down various signs coming from transfer factors. We present a shortcut to make it more manageable. The final result should be thought of as an ``Eichler-Selberg'' trace formula for quaternionic automorphic representations on $G_2$. As a last piece of the puzzle, section \ref{Hterm} uses results about level-$1$ forms from \cite{CR15} to reduce counts of forms on $H$ to counts of classical modular forms. 

Section \ref{JL} uses all these formulas to characterize representations in $\mc Q_k(1)$ with $k>2$ in terms of automorphic representations on $G_2^c$ and certain pairs of classical modular forms. We substitute in values for the $I_\spec^{G_2^c}$-terms from \cite{CR15} and present a final table of dimensions: table \ref{tabledims}, in section \ref{finalcount}. Finally, building off an impressive undergraduate thesis \cite{Sul13} of Sullivan, we give a relatively simple closed-form formula for the $G_2^c$-term and present the resulting closed-form formula for $|\mc Q_k(1)|$ in section \ref{explicit}.

\subsection{Acknowledgements}

This work was done under the support of NSF RTG grant DMS-1646385. I would like to thank Kimball Martin and Aaron Pollack for discussions at conferences that eventually led to the consideration of this problem. The key idea that something like Theorem \ref{qspec} could hold for $G_2$, thereby making this problem solvable grew out of email conversations with Sam Mundy. I thank Jeffrey Adams and David Vogan for the full argument of \ref{qspec}. Olivier Ta\"ibi provided a great deal of help in pointing out many, many tricks and previous results I could use to keep computations reasonably simple and actually feasible. In particular, I am grateful for the suggestion to use ``method 2'' described in section \ref{geomfinal} that he passed on from Ga\"etan Chenevier. I also thank Gordan Savin for pointing out the thesis of Sullivan which saved a lot of work in producing the closed-form formula. Finally, Alexander Bertoloni-Meli provided much help in teaching me things about transfer factors. The root system picture for $G_2$ was heavily modified from an answer by user Heiko Oberdiek on TeX Stack Exchange.

\subsection{Notation}
Here is a list of notation used throughout:

\noindent{The group $G_2$}
\begin{itemize}
\item $G_2$ is the standard exceptional Chevalley group defined over $\Z$. 
\item $G_2^c$ is the unique inner form of $G_2$ over $\Q$. Recall that $(G_2^c)_\R$ is the compact real form. 
\item $\alpha_i, \eps_i, \lb_i$ are particular roots and coroots of $G_2$ defined in section \ref{rootsG2}. 
\item $s_\lb$ is the simple reflection associated to root or coroot $\lb$. 
\item $\rho$ is half the sum of the positive roots of $G_2$. 
\item $V_\lb$ is the finite dimensional representation of $G_2$ of highest weight $\lb$. 
\item $K$ is a choice of maximal compact subgroup $\SU(2) \times \SU(2)/\pm 1$ of $G_2(\R)$. 
\item $K^\infty$ is the product of maximal compact subgroups $G_2(\Z_p)$ over all $p$. 
\item $\Om = \Om_\C$ is the Weyl group of $G_2$.
\item $\Om_\R$ is the Weyl group of $K$ as a subset of $\Om$. 
\item $H$ will often refer to the specific endoscopic group $\SL_2 \times \SL_2/\pm 1$ of $G_2$. 
\item $\mc Q_k(1)$ is the set of discrete, quaternionic automorphic representations of $G_2$ of weight $k$ and level $1$ (see section \ref{qdsdefg2}). 
\item $\pi_k$ is the weight $k$ quaternionic automorphic representations of $G_2$ (see section \ref{qdsdefg2}). 
\end{itemize}
\noindent{General groups}
\begin{itemize}
\item
$G_\infty = \Res^F_\Q G(\R)$ for $G$ a reductive group over number field $F$. Since most groups here are over $\Q$, $G_\infty$ is usually $G(\R)$. 
\item
$G^S, G_S$ are more generally the standard upper- and lower-index notation for $G(\A^S), G(\A_S)$---leaving out the places in $S$ or only including the places in $S$ respectively.
\item
$\Om(G)$ is the absolute Weyl group of $G$. 
\item
$\Om_\R(G)$ is the subset of the Weyl group of $G_\infty$ with respect to an elliptic maximal torus (if one exists) generated by elements of $G_\infty$. 
\item
$K_G$ is a maximal compact subgroup of $G_\infty$. 
\item
$K^\infty_G$ for unramified $G$ is the product of chosen hyperspecial subgroups at all finite places. 
\item
$\rho_G$ is half the sum of the positive roots of $G$. 
\item
$[G(F)], [G(F)]_\ssm, [G(F)]_\mathrm{st}, [G(F)]_\el$ are the (semisimple, stable, elliptic) conjugacy classes of $G(F)$. 
\end{itemize}
\noindent{Real test functions}
\begin{itemize}
\item
$\varphi_\pi$ for $\pi$ a discrete series representation of $G_\infty$ is the pseudocoefficient defined in the corollary to proposition 4 in \cite{CD90}. 
\item
$\Pi_\disc(\lb)$ is the discrete series $L$-packet corresponding to dominant weight $\lb$. 
\item
$\eta_\lb$ is the Euler-Poincar\'e function from \cite{CD90} for $\Pi_\disc(\lb)$. We normalize it to be the average of the pseudocoefficients for $\pi \in \Pi_\disc(\lb)$ instead of their sum. 
\end{itemize}
\noindent{Trace Formula}
\begin{itemize}
\item
$\mc{AR}(G)$, $\mc{AR}_\disc(G)$ is the set of (discrete) automorphic representations on $G$. 
\item
$\mc{AR}_\ur(G)$ for $G$ unramified is the space of unramified automorphic representations of $G$.
\item
$m_\disc(\pi), m_\cusp(\pi)$ are  the multiplicities of automorphic representation $\pi$ in the discrete (cuspidal) subspace.
\item
$I^G_\spec, I^G_\geom, I^G_\disc$ are the distributions from Arthur's invariant trace formula on $G$.
\item
$S^G = S^G_\geom$ is the stable distribution defined in theorem \ref{stableL2}.
\end{itemize}
\noindent{Miscellaneous} 
\begin{itemize}
\item $\1_S$ is the indicator function of set $S$. 
\item $\1_G$ is the trivial representation on group $G$. 
\item  $\mc S_k(1)$ is the set of normalized, classical, cuspidal eigenforms on $\GL_2$ of level $1$ and weight $k$. 
\end{itemize}

\section{Quaternionic Discrete Series}\label{qdsdef}
\subsection{Discrete Series}
\subsubsection{Parametrization}
For this section, let $G$ be a reductive group over $\R$ and $K$ a maximal compact of $G(\R)$. Assume $G$ has elliptic torus $T$ so that $G(\R)$ has discrete series. Without loss of generality, $T \subseteq K$. Recall the notation from \cite[\S2.2.1]{Dal22} to discuss discrete series. In particular, recall the two parametrizations of discrete series on $G(\R)$:
\[
\pi^G_{\lb, \om} = \pi^G_{\om(\lb + \rho_{G})}
\]
for $\lb$ a dominant (but possibly irregular) weight of $T$ and $\om$ a Weyl-element that takes a chosen $\Om_G$ dominant chamber into a chosen $\Om_K$-dominant one. Note that $\pi^G_{\lb, \om}$ has infinitesimal character $\lb + \rho_{G}$. Recall that $\om(\lb + \rho_{G})$ is called the \emph{Harish-Chandra parameter} of this discrete series.

\subsubsection{Their pseudocoefficients}
Recall from the corollary on page 213 in \cite{CD90} the notion of pseudocoefficients $\varphi_\pi$ for discrete series $\pi$. They are defined by their traces against standard modules $\rho$:
\[
\tr_\rho(\varphi_\pi) = \begin{cases}
1 & \rho = \pi \\
0 & \rho \text{ standard, } \sigma \neq \pi
\end{cases}.
\]
Here, a standard module is a parabolic induction of a discrete series or limit of discrete series. 
\begin{note}
By the Langlands classification, every irreducible representation has a \emph{character formula} writing it as a linear combination of standard modules in the Grothendieck group. By linearity of trace, if $\sigma$ is an irreducible representation, then $\tr_\sigma(\varphi_\pi)$ is the coefficient of $\pi$ in its character formula.
\end{note}

Recall also the Euler-Poincar\'e functions $\eta_\lb$ that we normalize to be the average of pseudocoefficients over an $L$-packet of infinitesimal character $\lb + \rho_G$. For a quick summary of relevant properties of these functions in the notation used here, see \cite[\S2.2.2]{Dal22}.

\subsection{Trace Distinguishability}
A priori, the trace against $\varphi_\pi$ may be non-zero for certain non-tempered representations in addition to just $\pi$. This could make $\varphi_k$ unusable as a test function to pick out just automorphic representations $\pi$ with $\pi_\infty = \pi$. We analyze when this happens. 

\begin{dfn}
Call discrete series $\pi$ on group $G(\R)$ \emph{trace-distinguishable} if for all unitary representations $\sigma$ of $G(\R)$
\[
\tr_\sigma(\varphi_k) = \begin{cases} 1 & \sigma = \pi_k \\ 0 & \text{else} \end{cases}.
\]
\end{dfn}

To motivate this definition, the Paley-Weiner theorem of \cite{CD90} shows that $\varphi_\pi$ is the only compactly supported function that could have the property of isolating $\pi$ in the unitary dual in this way---there are none if $\pi$ isn't trace-distinguishable.

\begin{prop}\label{td}
Let discrete series $\pi$ on $G(\R)$ have Harish-Chandra parameter $\xi$. Define
\[
S_\xi =  \{\alpha \in \Phi_G : \langle \xi, \alpha^\vee \rangle = 1\}. 
\]
where $\Phi_G$ is the set of roots of $(G, T)$ for $T$ elliptic. Then $\pi$ is trace-distinguishable if and only if $\pi$ contains no non-compact roots. 
\end{prop}

\begin{proof}
The following proof was described to me by David Vogan. Choose simple roots so that $\xi$ is dominant. By the same argument of Vogan described in \cite[lem. 6.3.1]{Dal22}, $\tr_\sigma(\varphi_\pi) = 0$ unless $\tr_\sigma(\eta_{\xi - \rho_G}) \neq 0$ for $\eta_{\xi - \rho_G}$ the Euler-Poincar\'e function at infinitesimal character $\xi$. If $\sigma$ is unitary, this is only possible if $\sigma$ has non-zero $(\mf g,K)$-cohomology with respect to the irreducible finite-dimensional representation of infinitesimal character $\xi$. 

By the main classification result of \cite{VZ84}, the only representations that do so are the discrete-series packet $\Pi_\lb(\xi - \rho_G)$ and certain cohomological inductions $A_{\mf q}(\lb)$ for $\theta$-stable parabolic subalgebras $\mf q$ of $\mf g$ and $\lb$ a character of the Levi algebra $\mf l$ associated to $\mf q$ (see, for example, \cite[\S2.1]{AJ87} for a definition of $A_\mf q(\lb)$).  It therefore suffices to show that none of these except $\pi$ itself have $\pi$ appearing in their character formulas. 

The only non-trivial case to check is that of non-discrete-series $A_{\mf q}(\lb)$. Theorem 8.2 in \cite{AJ87} provides its character formula and shows that the discrete series that appear are exactly those with Harish-Chandra parameters of the form $\lb + \om \rho_\mf l$ where $\om$ ranges over the Weyl group of $\mf l$. For each $\om$, pick a set of simple roots of $\mf l$ so that $\om \rho_l$ is dominant. Then for simple root $\alpha$ of $\mf l$, 
\[
\langle \lb + \om \rho_\mf l, \alpha \rangle = \langle \om \rho_\mf l, \alpha \rangle = 1.
\]
In particular, if $\pi$ appears in the character formula for $A_\mf q(\lb)$, then there is a choice of simple roots of $\mf l$ that are in $S_\xi$.

Finally, since $\lb$ is regular, for any root  $\alpha$ of $G$, $|\langle \lb, \alpha^\vee \rangle| \geq 1$. Therefore, $S_\xi$ needs to be a subset our simple roots chosen to make $\lb$ $G$-dominant. Let $\mf l_\xi$ be the associated Levi subalgebra. If $\pi$ appears in the character formula for $A_\mf q(\lb)$, the above gives that $\mf l \subseteq \mf l_\xi$. Therefore, if $S_\xi$ has no non-compact roots, then $\mf l$ is compact, so our condition on $\mf l$ implies that $A_\mf q(\lb)$ is discrete series (see, for example, the bottom of \cite[pg. 272]{AJ87}) and therefore equal to $\pi$. In total, $\pi$ cannot appear in other character formulas completing one direction. 

In the other direction, if $S_\xi$ has a non-compact root, then this can be used to construct a rank-$1$ Levi subalgebra $\mf l$ that isn't compact. Pick corresponding $\mf q$ and choose chamber for $\mf l$ so that $\lb$ is $\mf l$-dominant. Then $\pi$ will appear in the character formula of $A_\mf q(\lb - \rho_\mf l)$ which isn't discrete series.
\end{proof}

\subsection{Quaternionic Discrete Series}
Quaternionic discrete series are a special class of discrete series picked out in \cite{GW96}. We recall some needed definitions and properties:
\begin{dfn}
Call $G(\R)$ \emph{quaternionic} if $K$ is isogenous to a group of the form $\SU_2(\R) \times L$ (that has the same rank as $G$). 
\end{dfn}

\begin{dfn}
If $G(\R)$ is quaternionic, call discrete series $\pi$ \emph{quaternionic} if its minimal $K$ type lifts to a representation of the form $V \boxtimes \1_L$ on $\SU_2(\R) \times L$. Let the \emph{weight} of $\pi$ be $(\dim V -1)/2$. 
\end{dfn}

By looking at extended root diagrams:
\begin{lem}
Group $G(\R)$ is quaternionic if and only if there is a choice of simple roots of $(G(\R),T)$ such that there is a is unique non-compact simple root that also the unique simple root not perpendicular to the highest root. 
\end{lem}
Then, by Blattner's formula for minimal $K$-types:
\begin{lem}\label{qparam}
Let $G$ have quaternionic discrete series with simple roots chosen as in the previous lemma. Then all quaternionic discrete series have Harish-Chandra parameter of the form $n \beta' + \rho_G$ for $n \in \Z_{\geq 0}$ and $\beta'$ the highest root.
\end{lem}

Miraculously, almost all quaternionic discrete series are trace distinguishable:
\begin{prop}\label{qpstrace}
Let $\pi$ be a quaternionic discrete series of $G(\R)$ with infinitesimal character not equal to $\rho_G$. Then $\pi$ is trace-distinguishable. 
\end{prop}

\begin{proof}
If $\lb = n \beta' + \rho_G$ as in lemma \ref{qparam}, then $S_\lb$ from proposition \ref{td} is a subset of the simple roots chosen in lemma \ref{qparam}. Since $\beta'$ is not perpendicular to the unique non-compact simple root and $n > 1$, $S_\lb$ can only contain compact roots. 
\end{proof}

\section{Trace Formula}\label{geomG2}
Let $G$ be a reductive group over number field $F$ such that $G_\infty$ is quaternionic. 
\begin{dfn}
A \emph{quaternionic automorphic representation} on $G$ is an automorphic representation $\pi$ such that $\pi_\infty$ is quaternionic. 
\end{dfn} 
In this section we construct an ``explicit'' trace formula for studying almost all quaternionic automorphic representations. 

\subsection{Spectral Side}
The previous discussion on quaternionic discrete series shows:
\begin{thm}\label{qspec}
Let $G_\infty$ have quaternionic discrete series and let $\pi_0$ be a quaternionic discrete series of $G_\infty$ with infinitesimal character not equal to $\rho_G$. Then
\begin{align*}
I_\spec(\varphi_{\pi_0} \otimes f^\infty) & = \sum_{\pi \in \mc{AR}_\disc(G_2)} m_\disc(\pi) \delta_{\pi_\infty = \pi_0} \tr_{\pi^\infty} (f^\infty) \\
&= \sum_{\pi \in \mc{AR}_\cusp(G_2)} m_\cusp(\pi) \delta_{\pi_\infty = \pi_0} \tr_{\pi^\infty} (f^\infty). 
\end{align*}
\end{thm}

\begin{proof}
The statement for discrete representations is the same argument as \cite[prop. 6.3.3]{Dal22} after we know proposition \ref{qpstrace} that these quaternionic discrete series are trace-distinguishable. Since $\pi_\infty = \pi_0$ is necessarily discrete series for the non-zero terms, the main result of \cite{Wal84} shows that  $m_\cusp(\pi) = m_\disc(\pi)$. 
\end{proof}

\begin{note}
Of course, this theorem holds more generally for $\pi_0$ an arbitrary trace-distinguishable discrete series. 
\end{note}

\subsection{Geometric Side/The Hyperendoscopy Formula}
\subsubsection{Notation}
We will need to recall some extra notation related to general reductive group $H$ over $F$ to understand the geometric side
\begin{itemize}
\item $\Om^c_H$ is the Weyl group generated by compact roots at infinity.
\item $d(H_\infty)$ is the size of the discrete series $L$-packets of $H_\infty$. Alternatively, $d(H_\infty) = |\Om_H|/|\Om^c_H|$. 
\item $k(H_\infty)$ is the size of the group $\mf K = \ker(H^1(\R, T_\el) \to H^1(\R, G_\infty))$ that appears in the theory of endoscopy for $G_\infty$.
\item $q(H_\infty) = \dim(H_\infty/K_\infty Z_{H_\infty})$ where $K_\infty$ is a maximal compact subgroup of $H_\infty$. 
\item $H_\infty^*$ is the quasisplit inner form of $H_\infty$.
\item $\bar H_\infty$ is the compact form. If $H_\infty$ has an elliptic maximal torus, this is inner.
\item $e(H_\infty)$ is the Kottwitz sign $(-1)^{q(H^*_\infty) - q(H_\infty)}$. 
\item $[H:M] = [H:M]_F = \dim(A_M/A_G)$, where $A_\star$ is the maximal $F$-split torus in the center of $\star$. We call this the index of $M$ in $H$.  
\item $\tau(H)$ is the Tamagawa number of $H$.
\item $\Mot_H$ is the Gross motive for $H$.
\item $L(\Mot_H)$ is the value of the corresponding $L$-function at $0$ (or residue of the pole). 
\item $\iota^H(\gamma)  = \iota_F^H(\gamma)$ for $\gamma \in H(F)$ is the number of connected components of $H_\gamma$ that have an $F$-point.
\end{itemize}

\subsubsection{Preliminaries}
Let $\pi_0$ be a quaternionic discrete series on $G_\infty$. We will use the hyperendoscopy formula of \cite{Fer07} to compute $I_\geom(\varphi_{\pi_0} \otimes f^\infty)$.  We need to apply the general case of \cite[Thm. 4.2.3]{Dal22} since $G$ might have endoscopy without simply connected derived subgroup. We use notation from \cite[\S3,4]{Dal22} to discuss endoscopy and hyperendoscopy. See \cite{KS99} for a full reference to the theory of endoscopy and \cite{Lab11} for a course-notes-style introduction.

Let $\eta$ be the Euler-Poincar\'e function for the $L$-pacekt $\Pi_\disc(\lb)$ that conatins $\pi_0$. Let $\mc{HE}_\el(G)$ be the set of non-trivial hyperendoscopic paths for $G$. Then, in the notation of \cite[\S4]{Dal22},
\[
I^{G_2}_\geom(\varphi_{\pi_0} \otimes f^\infty) = I^{G_2}_\geom(\eta_k \otimes f^\infty) + \sum_{H \in \mc{HE}_\el(G_2)} \iota(G, \mc H) I^{\td{\mc H}}_\geom(((\eta  - \varphi_{\pi_0}) \otimes f^\infty)^{\td{\mc H}}),
\]
where the $\td{\mc H}$ are choices of $z$-pair paths when they are needed.

\subsubsection{Telescoping}
Next, an unpublished result of Kottwitz summarized in \cite[\S5.4]{Mor10} and proved by other methods in \cite{Pen19} stabilizes $I_\geom(\varphi \otimes f^\infty)$ when $\varphi$ satisfies a technical property of being stable-cuspidal (as EP-functions are but pseudocoefficients are not):

\begin{thm}\label{stableL2}
Let $\varphi$ be stable cuspidal (eg. an EP-function, but not a pseudocoefficient) on $G_\infty$ and $f^\infty$ a test function on $G(\A^\infty)$. Then
\[
I^{G}_\geom(\varphi \otimes f^\infty) = \sum_{H \in \mc E_\el(G)} \iota(G,H) S^{\td H}_\geom((\varphi \otimes f^\infty)^{\td H}),
\]
where $\mc E_\el(G)$ is the set of elliptic endoscopic groups for $G$ and the $\td H$ are $z$-extensions if necessary. The transfers $(\varphi \otimes f^\infty)^{\td H}$ depend on choices of measures for $G$ and $H$. 

The $S_\geom$ terms are defined by their values on Euler-Poincar\'e functions:
\begin{multline*}
S^H_\geom(\eta_\lb \otimes f^\infty)  = \sum_{M \in \mc L^\cusp(H)} (-1)^{[H : M]} \f{|\Om_{M,F}|}{|\Om_{H,F}|}  \tau(M) \\
 \times \sum_{\gamma \in [M(\Q)]_{\st, \el^\infty}} |\iota^M(\gamma)|^{-1} \f{e(\bar M_{\gamma, \infty})}{ \vol(\bar M_{\gamma, \infty}/A_{\bar M_\gamma, \infty})} \f{k(M_\infty)}{k(H_\infty)} \Phi^H_M(\gamma, \lb) SO_\gamma^\infty((f^\infty)_M),
\end{multline*}
choosing Tamagawa globally measure on all centralizers. The volume on $\bar M_{\gamma, \infty}$ is transferred from that on $M_{\gamma, \infty}$ in the standard way for inner forms so that the entire term doesn't depend on a choice of measure at infinity. 
\end{thm}



There's an alternating sign in the hyperendoscopy formula: if $\mc H$ is a hyperendoscopic path, then $-\iota(G, \mc H) \iota(\mc H, H) = \iota(G, (\mc H, H))$ for $H$ any endoscopic group of $\mc H$. Here, $(\mc H, H)$ represents the concatenation and $\mc H$ is overloaded to also refer to the last group in $\mc H$. 

In particular, substituting in the stabilization telescopes the hyperendoscopy formula.

\subsection{Final Formula and Usage Notes}\label{geomfinal}
\subsubsection{Formula}
Telescoping and adding in theorem \ref{qspec}, then we get final formula for quaternionic discrete series $\pi_0$ of $G_\infty$ that has infinitesimal character not equal to $\rho_G$:
\begin{multline}\label{mainformula}
\sum_{\pi \in \mc{AR}_\disc(G)} m_\disc(\pi) \delta_{\pi_\infty = \pi_0} \tr_{\pi^\infty} (f^\infty) \\ = S^{G}_\geom(\eta \otimes f^\infty) + \sum_{\substack{H \in \mc E_\el(G) \\ H \neq G}} \iota(G, H) S^{\td H}_\geom((\varphi_{\pi_0} \otimes f^\infty)^{\td H}).
\end{multline} \
The right side can be evaluated with theorem \ref{stableL2}.

We recall:
\[
\iota(G,H) = |\Lambda(H, \mc H, s, \eta)|^{-1} \f{\tau(G)}{\tau(H)},
\]
where $\Lambda(H, \mc H, s, \eta)$ is the image in $\Out(\wh H)$ of the automorphisms of the endoscopic quadruple. 

While getting a formula in terms of the distributions $S_\geom$ on smaller endoscopic groups comes immediately from stabilization, the above telescoping argument seems to be necessary to get explicit formulas for the $S_\geom$ when using a test factor at infinity that is just cuspidal instead of a stable cuspidal.

\subsubsection{Usage}
There are two possible methods to compute terms here. If we were interested in working with more general groups or at more general level, something like method 1 would have been necessary. However, our application case of level-$1$ representations on $G_2$ allows us to use the much easier method 2. Method 2 in fact does not even need an explicit expansion for $S_\geom$.

\noindent\underline{Method 1:}

We can try calculate the $S_\geom$ terms directly from their formula in theorem \ref{stableL2}. We will need to choose Euler-Poincar\'e measure at $\bar M_\gamma$ times canonical measure for the orbital integrals (canonical measure is the same for all inner forms). This adds an extra factor of 
\[
d(\bar M_{\gamma, \infty}) \f{L(\Mot_{M_{\gamma}})}{e(\bar M_{\gamma, \infty}) 2^{\rank(M_{\gamma, \infty})}}
\]
by \cite[lem. 6.2]{ST16}. Since $d(H_\infty) = 1$ and $\vol_{EP}(H_\infty /A_{H_\infty}) = 1$ for $H$ compact, this expands the terms in \eqref{mainformula} as:
\begin{multline*}
S^H_\geom(\eta_\lb \otimes f^\infty)  = \sum_{M \in \mc L^\cusp(H)} \lf( (-1)^{[H : M]} \f{|\Om_{M,F}|}{|\Om_{H,F}|} \ri) \lf( \tau(M) \f{k(M_\infty)}{k(H_\infty)} \ri)  \\
 \times \sum_{\gamma \in [M(\Q)]_{\st, \el^\infty}}  2^{-\rank(M_{\gamma, \infty})}   \Phi^H_M(\gamma, \lb) 
 \lf( L(\Mot_{M_{\gamma}}) |\iota^M(\gamma)|^{-1}  SO_\gamma^\infty((f^\infty)_M) \ri),
\end{multline*}
where the stable orbital integrals are now computed using canonical measure on centralizers.  

The hardest terms here are the stable orbital integrals, the $L$-values, and the characters $\Phi$. The constant terms $(f^\infty)_M$ are explicit integrals.

The $L$-values may be computed as products of values of Artin $L$-functions by explicitly describing the motives from \cite{Gro97}. The terms $\Phi$ can be reduced to linear combinations of traces of $\gamma$ against finite dimensional representations of $G$ by the algorithm on \cite[pg. 273]{Art89}. These can of be computed by the Weyl character formula and its extension to irregular elements stated in, for example, \cite[prop. 2.3]{CR15}. 

The stable orbital integrals unfortunately cause far more difficulty. For specific groups, including our eventual application case of $G_2$, they are computed and listed in tables in \cite[pg. 159]{GP05}. First, they are interpreted as orbital integrals on compact-at-infinity form $G^c$ by endoscopic transfer. The spectral side of the trace formula on $G^c$ is then possible to compute, allowing the orbital integrals to be solved for once the coefficients in terms of $L$-values are known. Alternatively, \cite{CT20} uses another trick, inputting vanishing results for small weight automorphic representations to solve for unstable orbital integrals in the resulting system of linear equations.
 
They can also be computed directly from unstable orbital integrals: \cite{Kot80} and \cite{Tai17} use Bruhat-Tits theory to do this for $\GL_3$ and some classical groups respectively. Either way, all currently known methods are not fully general, and extremely complicated. 

\noindent\underline{Method 2:}

Fortunately, there is a much simpler way to compute for our desired application of level-$1$ representations on $G_2$. Recalling that $I^{G_2^c}$ is known from \cite{CR15}, we can compare the expansions \eqref{mainformula} for $G_2$ and $G_2^c$. The term for $S^{G_2}$ a
appears in the expansion for $I^{G_2^c}$ and can therefore be solved for and substituted in the expansion for $I^{G_2}$. In total we get a formula
\[
I^{G_2} = I^{G_2^c} + \text{corrections},
\]
where the corrections are in terms of $S^H$ for smaller endoscopic $H$. 

In the next section we will see that there aren't actually that many $H$ appearing. Finally, section \ref{Hterm} will show that the terms for these $H$ are easily computed through another trick in the case of level-$1$. Method 2 also gives in section \ref{JL} a Jacquet-Langlands-style result comparing quaternionic representations on $G_2$ to representations on $G_2^c$.  

\begin{note}
We comment on possible extensions of method 2. The comparison to a compact form would work for any group with a form that is compact at infinity and unramified at all finite places. These appear in types, $G_2, B_3, D_4, B_4, F_4, B_7, D_8, B_8$, and  $E_8$ as enumerated in \cite{Gro96b}. 

Being able to easily count the endoscopic terms spectrally is more rare and requires some kind of recursive expansion down to only terms of Lie type $A_1^n$. This in particular works out for type $D_4$, so level-$1$ forms on type $D_4$ should be countable analogously to method 2. 

In another direction, plugging in other unramified test functions could compute counts weighted by Hecke eigenvalues. These would be in terms of the same weighted counts on $G_2^c$ and certain other weighted counts of classical modular forms that are determined by combinatorial formulas for unramified transfers as explained in \cite[\S5.4]{Dal22}. 
\end{note}

\section{$G_2$ Computation Set-up}\label{g2setup}
From now on, we specialize to $G = G_2$ and discuss how to apply the previous theory to count $|\mc Q_k(1)|$. 
\subsection{Root System of $G_2$}\label{rootsG2}
\subsubsection{Roots}
We use notation from \cite{LS93} to specify the root system of $G_2$. Let $K$ be the maximal compact $\SU(2) \times \SU(2)/\pm 1$ of $G_2(\R)$ and choose a maximal torus $T(\R)$ that is inside $K$. Make a choice of simple roots of $G_2(\R)$ that are non-compact, in this case determining a unique dominant chamber with respect to both $G_2$ and $K$. Let $\beta$ be the highest root of $G_2$ with respect to the choice of simple roots and note that it is long.

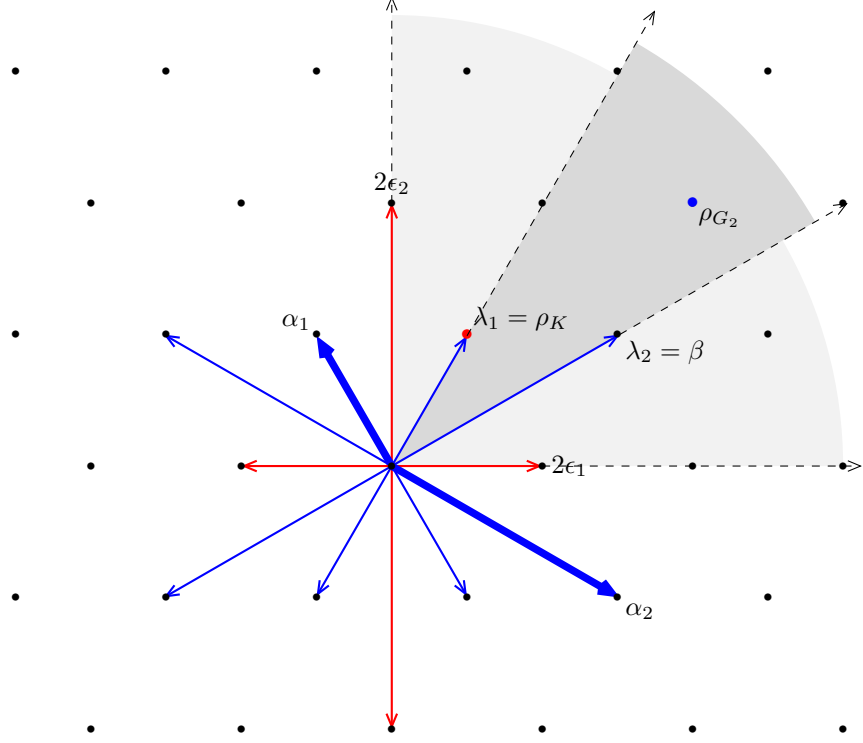
\begin{figure}
\centering
\caption{Character lattice, roots, and choices of dominant chamber for $G_2$}
  \begin{tikzpicture}[
    -{Straight Barb[bend,
       width=\the\dimexpr10\pgflinewidth\relax,
       length=\the\dimexpr12\pgflinewidth\relax]},
  ]
  
  \fill[color = gray!10] (0,0) --  (90:6) arc(90:60:6) -- cycle;
\fill[color = gray!10] (0,0) --  (30:6) arc(30:0:6) -- cycle;
\fill[gray!30] (0,0) --  (60:6.5) arc(60:30:6.5) -- cycle; 

      \draw[thick, red] (0, 0) -- (0*60:2);
      \draw[thick, blue] (0, 0) -- (30 + 0*60:3.5);
       \draw[thick, blue] (0, 0) -- (1*60:2);
      \draw[thick, red] (0, 0) -- (30 + 1*60:3.5);
       \draw[line width = 1mm, blue] (0, 0) -- (2*60:2);
      \draw[thick, blue] (0, 0) -- (30 + 2*60:3.5);
       \draw[thick, red] (0, 0) -- (3*60:2);
      \draw[thick, blue] (0, 0) -- (30 + 3*60:3.5);
       \draw[thick, blue] (0, 0) -- (4*60:2);
      \draw[thick, red] (0, 0) -- (30 + 4*60:3.5);
       \draw[thick, blue] (0, 0) -- (5*60:2);
      \draw[line width = 1mm, blue] (0, 0) -- (30 + 5*60:3.5);
    \node[right] at (0:2) {$2 \eps_1$};
     \node[above] at (90:3.5) {$2 \eps_2$};
    \node[below right, inner sep=.2em] at (30:3.5) {$\lb_2 = \beta$};
    \node[above left, inner sep=.2em] at (120:2) {$\alpha_1$};
    \node[above right, inner sep=.2em] at (60:2) {$\lb_1 = \rho_K$};
    \node[below right, inner sep=.2em] at (330:3.5) {$\alpha_2$};
    
    \foreach \i in {-4,-2, ...,6}{
    \foreach \j in {-2, 0, ..., 2}{
    \node at (\i, 1.75*\j) {\tiny \textbullet};
    }
    } 
    
    \foreach \i in {-5,-3, ..., 5}{
    \foreach \j in {-1, 1, ..., 3}{
    \node at (\i, 1.75*\j) {\tiny \textbullet};
    }
    } 
    
    \node[blue] at (4, 3.5) {\small \textbullet};
    \node[red] at (1, 1.75) {\small \textbullet};
    
    \node[below right, inner sep=.2em] at (4, 3.5) {$\rho_{G_2}$};

    
    \draw [dashed] (90:3.5) -- (90:6.25);
    \draw [dashed] (60:2) -- (60:7);
    \draw [dashed] (30:3.5) -- (30:7);
    \draw [dashed] (0:2) -- (0:6.25);
  \end{tikzpicture}
  \label{G2rootpicture}
  \end{figure}

We now give explicit coordinates. As a mnemonic convention, roots indexed $1$ will be short and roots indexed $2$ will be long. Figure \ref{G2rootpicture} displays all the roots and shades our choices of dominant Weyl chambers for both $G_2$ and $K$. The compact roots at infinity are the four along the $\eps_i$-coordinate axes.

If the roots of the short and long $\SU_2$ are $2\eps_1$ and $2\eps_2$ respectively, then the simple roots of $G_2$ are: 
\[
\text{(short) }\alpha_1 = - \eps_1 + \eps_2,  \qquad \text{(long) } \alpha_2 = 3 \eps_1 - \eps_2.
\]
The other positive roots are:
\begin{gather*}
\text{(short) } 2 \eps_1 = \alpha_1 + \alpha_2, \qquad \eps_1 + \eps_2 = 2 \alpha_1 + \alpha_2, \\
\text{(long) } 2 \eps_2 = 3 \alpha_1 + \alpha_2,  \qquad 3 \eps_1 + \eps_2 = 3 \alpha_2 + 2 \alpha_2.
\end{gather*}
The fundamental weights are:
\[
\lb_1 = 2 \alpha_1 + \alpha_2, \qquad \lb_2 = 3 \alpha_1 + 2 \alpha_2.
\] 
Of course $\beta = \lb_2$. 

The Weyl group is generated by simple reflections:
\[
s_{\alpha_1} \begin{pmatrix} 2\eps_1 \\ 2\eps_2 \end{pmatrix} = \begin{pmatrix} \eps_1 + \eps_2  \\ 3\eps_1 - \eps_2 \end{pmatrix}, \qquad s_{\alpha_2} \begin{pmatrix} 2\eps_1 \\ 2\eps_2 \end{pmatrix} = \begin{pmatrix} -\eps_1 + \eps_2  \\ 3\eps_1 + \eps_2 \end{pmatrix}.
\]
Finally:
\begin{gather*}
\rho_K = \eps_1 + \eps_2 = 2 \alpha_1 + \alpha_2, \\
\rho_G = 4 \eps_1 + 2 \eps_2 = 5 \alpha_1 + 3 \alpha_2.
\end{gather*}

\subsubsection{Coroots}
Coroots will follow the opposite mnemonic: coroots indexed $1$ are long and coroots indexed $2$ are short.

Since $G_2$ has trivial center, $X^*(T)$ is the root lattice, which is exactly
\[
X^*(T) = \{a \eps_1 + b \eps_2 : a,b \in \Z, a + b \in 2\Z\}.
\]
Let $(\delta_1, \delta_2)$ be the dual basis to $(2\eps_1, 2\eps_2)$: i.e. $(\delta_i, \eps_j) = 1/2 \1_{i=j}$. Then:
\[
X_*(T) = \{a \delta_1 + b \delta_2 : a,b \in \Z, a + b \in 2\Z\}.
\]
Since $\eps_1$ and $\eps_2$ are perpendicular:
\begin{gather*}
(2 \eps_1)^\vee = 2\delta_1, \\
(2 \eps_2)^\vee = 2\delta_2.
\end{gather*}
More generally, the Weyl action gives:
\begin{gather*}
(\alpha_1^\vee, 2 \eps_1) = -1, \qquad (\alpha_1^\vee, 2\eps_2) = 3,\\
(\alpha_2^\vee, 2 \eps_1) = 1, \qquad (\alpha_2^\vee, 2\eps_2) = - 1,
\end{gather*}
so we get simple coroots:
\begin{gather*}
\alpha_1^\vee = -\delta_1 + 3\delta_2, \\
\alpha_2^\vee = \delta_1 - \delta_2.
\end{gather*}
This reproduces that the coroot lattice is $X_*(T)$, implying that $G_2$ is simply connected. For completeness:
\begin{gather*}
\lb_1^\vee = \delta_1 + 3 \delta_2,\\
\lb_2^\vee =  \delta_1 + \delta_2.
\end{gather*}

\subsection{Quaternionic Discrete Series for $G_2$}\label{qdsdefg2}
The quaternionic discrete series on $G_2$ of weight $k$ for $k \geq 2$ lies in the $L$-packet 
\[
\Pi^{G_2}_\disc((k-2) \beta).
\]
The members of this $L$-packet have Harish-Chandra parameters:
\[
(k-2)\beta + \rho_G, \qquad s_{\alpha_1}((k-2)\beta + \rho_G), \qquad s_{\alpha_2}((k-2)\beta + \rho_G).
\]
As in \cite{GGS02}, the quaternionic member is the one with minimal $K$-type $\lb_B = 2k\eps_2$. We know that the discrete series $\pi(\om, \lb)$ has minimal $K$-type
\[
\lb_B = \om(\lb + 2 \rho_G) - 2 \rho_K
\]
by the Blattner formula \cite[Thm. 9.20]{Kna01}. Therefore the weight-$k$ quaternionic discrete series $\pi_k$ is specifically $\pi(s_{\alpha_2}, (k-2)\beta)$---computing, $s_{\alpha_2}$ fixes $\rho_K$ so 
\[
s_{\alpha_2}(\lb + 2 \rho_G) - 2 \rho_K = s_{\alpha_2}(\lb + 2 \rho_G - 2 \rho_K) = s_{\alpha_2}(\lb + 2 \beta)  = s_{\alpha_2}(k \beta) = 2k \eps_2.
\]
This is the discrete series with Harish-Chandra parameter 
\[
\lb_{k,H}: = s_{\alpha_2}((k-2)\beta + \rho_G).
\]
Call it $\pi_k$ and its pseudocoefficient $\varphi_k$. 

Theorem \ref{qspec} then gives that for $k > 2$, 
\begin{equation}\label{ssideG2}
|\mc Q_k(1)| = I_\spec(\varphi_k \otimes \1_K)
\end{equation}
if we choose measures so that $\vol G_2(\wh \Z) = 1$. Note again that this heavily depends on the miracle of proposition \ref{qpstrace} and a similar result does \emph{not} hold either for pseudocofficients for the other members of $\Pi_\disc((k-2)\beta)$ or for the Euler-Poincar\'e function. 

\begin{note}
Theorem \ref{qspec} for just the case of $G_2$ can be produced much more easily by the computation in \cite{Mun20} of the $A$-packets of infinitesimal character $(k-2)\beta + \rho_G$ for $k>2$. Mundy found that $\pi_k$ appears in all of them. Therefore, trace-distinguishability follows immediately from \cite[Lemma 8.8]{AJ87} that a given discrete series appears in the character formula of exactly one element of such an $A$-packet. 
\end{note}

\section{Groups Contributing and Related Constants}\label{gc}
\subsection{Elliptic Endoscopy of $G_2$}\label{endoscopy}
The elliptic endoscopic groups of $G_2$ are $G_2$, $\PGL_3$ and $\SO_4$. This is stated in a thesis \cite{Alt13} but not fully explained, so we fill in some details here for reader convenience. We use notational conventions for endoscopy as in \cite[\S3]{Dal22}. 
We compute the possible endoscopic pairs $(s, \rho)$. 

Since $G_2$ has trivial center, the cohomology condition on $s$ is always satisfied so we don't bother checking it. Trivial center further gives that the isomorphism class of the pair only depends on $s$ through its centralizer. Next, we note a result that we thank a referee for pointing out:
\begin{lem}
Let $G$ be a split, simple, and simply connected group over number field $F$. Then all its endoscopic groups are split.
\end{lem}

\begin{proof}
Since $G$ is simply connected, there is an $L$-embedding $\Ld H \into \Ld G$. Since $G$ is split, there is a projection $\Ld G \onto \wh G$. By inspecting the reconstruction of $\Ld H$ from $(s, \rho)$, the image of $\Ld H$ in $\wh G$ is connected if and only if $\rho$ is trivial. 

However, since $\wh G$ also has trivial center, $s$ is fixed by $\rho$ which implies that the image of $\Ld H$ in $\wh G$ is the centralizer of $s$. It is therefore necessarily connected since $s$ is semisimple and $\wh G$ is simply connected.
\end{proof}

In particular $\rho$ is always trivial and we can always find a valid $s$ for any possible centralizer $\wh H$. The possible elliptic $\wh H$ with trivial $\rho$ are $G_2, \SL_3$ and $\SL_2 \times \SL_2/\{\pm 1\}$ corresponding to split endoscopic groups $G_2, \PGL_3$, and $\SL_2 \times \SL_2/\{\pm 1\}$. In each of these cases, $\Lambda = 1$.

If a group contributes to the stabilization applied to our test function, then by the fundamental lemma, it needs to be unramified away from infinity. By formulas for transfers of pseudocoefficients (see \cite[lem. 5.6.1]{Dal22}), it also needs to have an elliptic maximal torus at infinity. The only groups contributing are therefore the $G_2$ and the $\SL_2 \times \SL_2/\{\pm 1\}$.

\subsection{Endoscopic Constants and Normalizations}

 \subsubsection{The $\iota$} Let $H = \SL_2 \times \SL_2/\pm 1$ and let $G_2^c$ be the unique non-split inner form of $G_2$ over $\Q$ which is compact at infinity. 
Then,
\begin{gather*}
\iota(G_2^c, H) = \iota(G_2, H) = |\Lambda(H, \mc H, s, \eta)|^{-1} \f{\tau(G)}{\tau(H)} = 1 \cdot \f12, \\
\iota(G_2^c, G_2) = 1,
\end{gather*}
by Kottwitz's formula for Tamagawa numbers (note that $\ker^1(\Q, Z_H) = \ker^1(\Q, \{\pm 1\})  =1$).

\subsubsection{The transfer factors} \label{transferfactors} We also need to fix transfer factors at all places for both $G_2$ and $G_2^c$ to compute transfers. The computations in \cite{Tai17} demonstrate how to do so explicitly. First, they can be chosen consistently by fixing a global Whittaker datum. The corresponding local Whittaker datum determines the local transfer factors on $G_2$ as in \cite{KS99}. Since $G_2$ is defined over $\Z$, we can choose global data that is unramified/admissible at all finite places with respect to the $G_2(\Z_p)$ so we can use the fundamental lemma at all finite places as in \cite[\S7]{Hal93}.  By \cite[\S4.4]{Kal18}, the Whittaker datum on $G_2$ also gives compatible local transfer factors on $G_2^c$ (note that $G_2$ is simply connected for applying theorem 4.4.1). These allow us to use the fundamental lemma at finite places since transfer factors there stay the same as our choices for $G_2$.

We need to know two things about the Archimedean transfer factors. First, Whittaker normalization lets us use the formulas for discrete transfer from \cite{She10} on both $G_2$ and $G_2^c$. Note here that $G_2^c$ is in particular a pure inner form since $G_2$ is adjoint. This formula is stated in a slightly easier to use form in \cite[\S IV.3]{Lab11} for our case of $\rho_G - \rho_H \in X^*(T)$ (the $\inv(\pi(1), \pi(w))$ of Shelstad is the $\kappa \cdot \eps$ term of Labesse). 

Second, we need to know which element of $\Pi_\disc((k-2)\beta)$ our Archimedean Whittaker datum makes Whittaker-generic. This will have to be $\pi_{(k-2)\beta, 1}$ since our choice of dominant Weyl chamber has all simple roots non-compact and is the only possible such choice up to $\Om_K$ (see the discussion before lemma 4.2.1 in \cite{Tai17}. In fact, there is only one possible conjugacy class of Whittaker datum at infinity by considerations explained there). 

\subsubsection{The stabilizations}
We fix canonical measure at finite places so that the fundamental lemma directly gives $\1_{K^\infty_{G_2}}^H = \1_{K^\infty_H}$. Recall that EP-functions and pseudocoefficients are defined depending on measure so we don't need to fix measure at infinity. 

Then, \eqref{mainformula} gives
\begin{multline}\label{G2form}
I^{G_2}(\varphi_{\pi_{G_2}(s_{\alpha_2}, (k-2) \beta)} \otimes \1_{K^\infty_{G_2}}) \\= S^{G_2}(\eta^{G_2}_{(k-2)\beta} \otimes \1_{K^\infty_{G_2}}) + \f12 S^H((\varphi_{\pi_{G_2}(s_{\alpha_2}, (k-2) \beta)})^H \otimes \1_{K^\infty_{H}}).
\end{multline}
A simple case of the discrete transfer formula in \cite[\S IV.3]{Lab11} computes that $(\eta^{G_2^c}_{(k-2)\beta})^{G_2} = \eta^{G_2}_{(k-2)\beta}$ (note that $\Om_\R(G_2^c) \bs \Om_\C(G_2^c)$ is trivial so $\kappa$ is too), so
\[
I^{G_2^c}(\eta^{G_2^c}_{(k-2)\beta} \otimes \1_{K^\infty_{G_2^c}}) =   S^{G_2}(\eta^{G_2}_{(k-2)\beta} \otimes \1_{K^\infty_{G_2}}) + \f12 S^{H}((\eta^{G^c_2}_{(k-2)\beta})^H \otimes \1_{K^\infty_{H}}).
\]
Since type $A_1 \times A_1$ has no non-trivial centralizer of full semisimple rank, all elliptic endoscopy of $\SL_2 \times \SL_2/\pm 1$ is non-split. Therefore, it is ramified at some prime, so the transfers of  $\1_{K^\infty_H}$ vanish, implying that $S^H = I^H$ on our test functions. Substituting one stabilization into another finally gives:
\begin{multline}\label{stabilization}
I^{G_2}(\varphi_{\pi_{G_2}(s_{\alpha_2}, (k-2) \beta)} \otimes \1_{K^\infty_{G_2}}) = I^{G_2^c}(\eta^{G_2^c}_{(k-2)\beta} \otimes \1_{K^\infty_{G_2^c}})  \\- \f12 I^H ((\eta^{G^c_2}_{(k-2)\beta})^H \otimes \1_{K^\infty_{H}}) + \f12 I^H((\varphi_{\pi_{G_2}(s_{\alpha_2}, (k-2) \beta)})^H \otimes \1_{K^\infty_{H}})
\end{multline}
under canonical measure at finite places. 

This is our realization of method 2.  There are three steps remaining to get counts:
\begin{enumerate}
\item
Compute the transfers of EP-functions to $H$. 
\item
Write the resulting $I^H(\eta_\lb \otimes \1_{K_H})$ terms in terms of counts of level-$1$, classical modular forms.
\item
Look up values for the $G_2^c$-term from \cite{CR15}. 
\end{enumerate}

\section{Real Endoscopic Transfers}\label{Rtrans}
Let $H$ again be the one endoscopic group we care about: $\SL_2 \times \SL_2/\{\pm1\}$. We want to compute $(\varphi_{\pi_{G_2}(s_{\alpha_2}, (k-2) \beta)})^H$ and $(\eta^{G_2^c}_{(k-2)\beta})^H$. By the choices of transfer factors in section \ref{transferfactors}, we may do so by the formulas in \cite[\S IV.3]{Lab11}. 

As a choice for computation that doesn't affect the final result, we realize the roots of $H$ as $2 \eps_1$ and $2 \eps_2$. Orient $X^*(T)$ by setting the 1st quadrant in $\eps_1$ and $\eps_2$ to be $H$-dominant. The Weyl elements $\Om(G, H)$ that send the $G$-dominant chamber to an $H$-dominant one are $\{1, s_{\alpha_1}, s_{\alpha_2}\}$. 



\subsection{Root Combinatorics}
Since $\rho_G - \rho_H \in X^*(T)$, \cite[\S IV.3]{Lab11} gives the transfer of the pseudocoefficient of the quaternionic discrete series to $H$:
\begin{multline}\label{transfer1}
(\varphi_{\pi_{G_2}(s_{\alpha_2}, (k-2) \beta)})^H =\\ \kappa^H(s_{\alpha_2}^{-1}) \eta^H_{(k-2)\beta + \rho_G - \rho_H} - \kappa^H(s_{\alpha_1} s_{\alpha_2}^{-1}) \eta^H_{s_{\alpha_1}((k-2)\beta + \rho_G) - \rho_H} - \eta^H_{s_{\alpha_2}((k-2)\beta + \rho_G) - \rho_H}
\end{multline}
for some signs $\kappa^H(\cdot)$. 

We compute that $\rho_H = \eps_1 + \eps_2$. Then
\[
(k-2)\beta + \rho_G - \rho_H = (k-2)(3 \eps_1 + \eps_2) + (3 \eps_1 + \eps) = 3(k-1) \eps_1 + (k-1) \eps_2.
\]
In addition,
\begin{align*}
s_{\alpha_1} \rho_G = 5\eps_1 + \eps_2, &\qquad s_{\alpha_1} \beta = \beta, \\
s_{\alpha_2} \rho_G = \eps_1 + 3\eps_2, &\qquad s_{\alpha_2} \beta = 2 \eps_2,
\end{align*}
so
\[
s_{\alpha_1}((k-2)\beta + \rho_G) - \rho_H = (k-2)(3 \eps_1 + \eps_2) + (4 \eps_1) = (3k-2) \eps_1 + (k-2) \eps_2
\]
and
\[
s_{\alpha_2}((k-2)\beta + \rho_G) - \rho_H = (k-2)(2\eps_2) + (2 \eps_2) = 2(k-1)\eps_2.
\]

\subsection{Endoscopic Characters}\label{echarcomp}
\subsubsection{Setup}
It remains to compute the $\kappa$ terms in \ref{transfer1}. These signs depend in a very complicated way on the realization of $H$ and the exact transfer factors chosen. We will therefore use an indirect trick to compute them more easily. 

Let $\psi_H$ be a (discrete in our case) $L$-parameter for $H(\R)$ and $\psi_G$ the composition with $\Ld H \into \Ld G_2$. Then we have an identity of traces over $L$-packets:
\[
S \Theta_{\psi_H}(f^H) = \sum_{\pi \in \Pi_{\psi_G}} \langle s_H, \pi \rangle \Theta_\pi(f),
\]
where $f^H$ is a transfer of $f$, $\Theta_\pi$ is the Harish-Chandra character, $S\Theta_{\psi_H}$ is the stable character corresponding to the $L$-packet, $\Pi_{\psi_G}$ is the $L$-packet corresponding to the $L$-parameter, and $\langle s_H, \pi \rangle$ is shorthand for a sign depending on $\pi$ and the choice of Whittaker data. This sign comes from pairing an element $s_H$ of the centralizer of $\psi_G$ determined by $H$ with a character associated to $\pi$ through the Whittaker datum. See \cite[\S1]{Kal16} for an exposition of how this works in general.

If $\pi$ on $G_2$ is discrete series, Labesse's formula tells us that we can choose:
\[
(\varphi^{G_2}_\pi)^H = \sum_\lb \eps(\lb, \pi) \eta^H_\lb
\]
for some signs $\eps$ that depend on the transfer factor and some set of weights $\lb$ that only depends on the infinitesimal character of $\pi$. 

Let $\psi_\lb$ be the $L$-parameter corresponding to weight-$\mu$ discrete series on $H$. Plugging this formula for $\varphi^{G_2}_\pi$ into the trace identity for $\psi_\lb$ gives that 
\[
\eps(\lb, \pi) =  \sum_\mu \eps(\mu, \pi) S \Theta_{\psi_\lb}(\eta_\mu^H) = \sum_{\pi' \in \Pi_{\psi_G}} \langle s_H, \pi' \rangle \Theta_{\pi'}(\varphi^{G_2}_\pi) = \begin{cases} \langle s_H, \pi \rangle & \pi \in \Pi_{\psi_G} \\ 0 & \text{else} \end{cases},
\]
where $\psi_G$ is the pushforward of $\psi_\lb$. The last equality is the definition of pseudocoefficient since all $\pi$ in a packet for an $L$-parameter should be tempered. This computation shows that  $\psi_\lb$ is required to push forward to the parameter for $\pi$ and that $\eps(\lb, \pi) = \langle s_H, \pi \rangle$. 

\subsubsection{The trick}
Now we are ready to compute the signs. Instead of doing the hard work of figuring out how the transfer factor directly affects the signs in Labesse's formulation, we will use the key fact that $\eps(\lb, \pi) = \langle s_H, \pi \rangle = 1$ whenever $\pi$ is the Whittaker-generic member of its $L$-packet. Therefore, in Labesse's formula for the \emph{generic} member $\pi_{1, (k-2)\beta}$,
\begin{multline*}
(\varphi_{\pi_{G_2}(1, (k-2) \beta)})^H = \eta^H_{(k-2)\beta + \rho_G - \rho_H}  + \kappa^H(s_{\alpha_1}) \sgn(s_{\alpha_1}) \eta^H_{s_{\alpha_1}((k-2)\beta + \rho_G) - \rho_H}\\ + \kappa^H(s_{\alpha_2}) \sgn(s_{\alpha_2}) \eta^H_{s_{\alpha_2}((k-2)\beta + \rho_G) - \rho_H},
\end{multline*}
all the coefficients need to be $1$. The allows to solve
\[
\kappa^H(s_{\alpha_1}) = \kappa^H(s_{\alpha_2}) = -1
\]
for our choice of transfer factors. Right-$\Om_\R$-invariance of Labesse's $\kappa$ then also gives that
\[
\kappa^H(s_{\alpha_1}s_{\alpha_2}) = -1. 
\]

\subsection{Final Formulas for Transfers}
Therefore, our final transfer is
\begin{equation}\label{pstransferG2}
(\varphi_{\pi_{G_2}(s_{\alpha_2}, (k-2) \beta)})^H  = -\eta^H_{3(k-1) \eps_1 + (k-1) \eps_2} + \eta^H_{(3k-2) \eps_1 + (k-2) \eps_2} - \eta^H_{2(k-1)\eps_2}.
\end{equation}
Transfers from $G_2^c$ are easier. Here, $\Om_\R(G_2^c) \bs \Om_\C(G_2^c)$ is trivial so the average value of $\kappa$ is $1$. Averaging Labesse's formula as in \cite[cor. 5.1.5]{Dal22} therefore gives:
\begin{equation}\label{eptransferG2}
(\eta^{G_2^c}_{(k-2) \beta})^H  = \eta^H_{3(k-1) \eps_1 + (k-1) \eps_2} - \eta^H_{(3k-2) \eps_1 + (k-2) \eps_2} - \eta^H_{2(k-1)\eps_2}.
\end{equation}




\section{The $H = \SL_2 \times \SL_2 / \pm 1$ term}\label{Hterm}
Here we compute the terms $I^H(\eta_\lb \otimes \1_{K_H})$ for Euler-Poincar\'e functions $\eta_\lb$.  Any $\lb = a \eps_1 + b \eps_2$ is a weight of $H$ if $a+b$ is even. Note first that
\[
I^H(\eta^H_\lb \otimes \1_{K_H}) = \sum_{\pi \in \mc{AR}_\disc(H)} \tr_{\pi_\infty}(\eta^H_\lb) \tr_{\pi^\infty}(\1_{K_H}) = \sum_{\substack{\pi \in \mc{AR}_\disc(H) \\ \pi \text{ unram.}}} \tr_{\pi_\infty}(\eta^H_\lb),
\]
by Arthur's simple trace formula and using our choice of canonical measure at finite places.

To move forward, we need to understand automorphic reps on $H$ by relating them to other groups. Consider the sequence
\[
1 \to \pm 1 \to \SL_2 \times \SL_2 \to H \to 1.
\]
It induces on local or global $F$:
\[
1 \to \pm 1 \to \SL_2 \times \SL_2(F) \to H(F) \to F^\times/(F^\times)^2 \to 1,
\]
using that $H^1(F, \pm 1) = F^\times/(F^\times)^2$ and $H^1(F, \SL_2) = 1$ for the $F$ we care about (the $\R$ case of the second equality comes from the determinant exact sequence on $\GL_2$). The image of $\SL_2 \times \SL_2(F)$ is the connected component $H(F)^0$.

\subsection{Cohomological Representations of $H(\R)$}
Next, we recall that the infinite trace measures an Euler characteristic against $(\mf h, K_{H,\infty})$-cohomology:
\[
\tr_{\pi_\infty}(\eta^H_\lb) = \chi(H^*(\mf h, K_{H,\infty}, \pi_\infty \otimes V_\lb)),
\]
where $\mf h$ is the Lie algebra of $H_\infty$ and $V_\lb$ is the finite dimensional representation of weight $\lb$ of $H^0_\infty$ pulled back to $H_\infty$.  Using the definition from \cite[\S5.1]{BW00},
\[
H^*(\mf h, K_{H,\infty}, \pi_\infty \otimes V_\lb) = H^*(\mf h, K^0_{H,\infty}, \pi_\infty \otimes V_\lb)^{K_{H,\infty}/K^0_{H,\infty}},
\]
it suffices to consider the $\pi_\infty$ whose restrictions to $H_\infty^0$ contain a component that is cohomological when pulled back to $[\SL_2 \times \SL_2](\R)$. By Frobenius reciprocity and semisimplicity of inductions, these are exactly the irreducible constituents of $\Ind_{H_\infty^0}^{H_\infty} \pi'$ for $\pi'$ cohomological of $H_\infty^0$.

Next, $H_\infty^0$ is index $2$ in $H_\infty$. Pick $h \in H_\infty - H_\infty^0$ and let $\pi'^{(h)}$ be the representation $\gamma \mapsto \pi'(h^{-1} \gamma h)$. Define character
\[
\chi : H_\infty \mapsto H_\infty/H_\infty^0 \simeq \{\pm 1\}. 
\]
As noted in the proof of lemma 2.5 in \cite{LL79}, there are two cases for $H_\infty^0$-representations $\pi'$:
\begin{enumerate}
\item
$\pi' \neq \pi'^{(h)}$: then $\Ind_{H_\infty^0}^{H_\infty} \pi'$ is irreducible and $\Res_{H_\infty^0}^{H_\infty}\Ind_{H_\infty^0}^{H_\infty} \pi' = \pi' \oplus \pi'^{(h)}$.
\item
$\pi' = \pi'^{(h)}$: then $\Ind_{H_\infty^0}^{H_\infty} \pi' = V \oplus (V \otimes \chi)$ for some irreducible $V$. 
Also, $\Res_{H_\infty^0}^{H_\infty}\Ind_{H_\infty^0}^{H_\infty} \pi' = \pi' \oplus \pi'$
\end{enumerate}

Recalling a standard result, the cohomological representations of $\SL_2(\R)$ with respect to $\lb$ are:
\begin{itemize}
\item
A discrete series $L$-packet $\{\pi_{\lb,1}$, $\pi_{\lb, s}\}$ (where $\Om_{\SL_2} = \{1, s\}$),
\item
The trivial representation $\1_{\SL_2}$ if $\lb = 0$. 
\end{itemize}
By the K\"unneth rule, cohomological representations of $\SL_2 \times \SL_2(\R)$ are exactly products of those on $\SL_2(\R)$.  Those of $H_\infty^0$ are exactly those of $\SL_2 \times \SL_2(\R)$ that are trivial on $\pm 1$---in other words, with $\lb = a\eps_1+b\eps_2$ and $a+b$ even.

Consider such $\lb$. There are three cases of inductions to consider to compute the cohomological representations of $H_\infty$. Note that conjugation by $h \in H_\infty -H_\infty^0$ swaps the two members of a discrete-series $L$-packet of an embedded $\SL_2$ factor and fixes the trivial representation. 
\begin{itemize}
\item
$a,b \neq 0$: We look at the inductions of products of discrete series. This is case (1) so the $4$ products pair up in sums that are $2$ members of an $L$-packet. These are of course $\pi^H_{\lb,1}$ and $\pi^H_{\lb, s}$ where $s$ is a length-$1$ element of $\Om_H$:
\begin{align*}
\pi^H_{\lb,1}|_{H_\infty^0} &= (\pi_{a \eps_1,1} \boxtimes \pi_{b \eps_2,1}) \oplus (\pi_{a \eps_1,s} \boxtimes \pi_{b \eps_2,s}), \\
 \pi^H_{\lb,s}|_{H_\infty^0} &= (\pi_{a \eps_1,1} \boxtimes \pi_{b \eps_2,s}) \oplus (\pi_{a \eps_1,s} \boxtimes \pi_{b \eps_2,1}).
\end{align*}
\item
Without loss of generality, $a = 0, b \neq 0$: We also need to consider inductions of $\1 \boxtimes \pi_{b \eps_2, \star}$. This is case (1) and both induce to a single irreducible $\sigma^H_{\lb}$:
\[
\sigma^H_\lb|_{H_\infty^0} = (\1 \boxtimes \pi_{b \eps_2,1}) \oplus (\1 \boxtimes \pi_{b \eps_2, s}).
\]
\item
$a = b = 0$: In addition to both the above, we need to consider the induction of $\1_{\SL_2} \boxtimes \1_{\SL_2}$. This is case (2). This trivial representation induces to $\1_{H_\infty} \oplus \chi$ on $H_\infty$. Both factors are cohomological.  
\end{itemize}

Grothendieck group relations stay true restricted to $H_\infty^0$ so we can compute traces against $\eta_\lb$. Recall that in $\SL_2(\R)$:
\[
\1 = I - \pi_{0, 1} - \pi_{0,s},
\]
where $I$ is some parabolically induced representation with trivial trace against $\eta^{\SL_2}_0$. 

First, by our normalization
\[
\tr_{\pi^H_{\lb,1}}(\eta^H_\lb) = \tr_{\pi^H_{\lb,s}}(\eta^H_\lb) = 1/2.
\]
Next, working in $H_\infty^0$ and for $\lb = b \eps_2$:
\[
\1 \boxtimes \pi_{\lb, \star} = ( I - \pi_{0, 1} - \pi_{0,s}) \boxtimes \pi_{\lb,\star} = I \boxtimes \pi_{\lb, \star} - \pi_{0, 1} \boxtimes \pi_{\lb, \star} - \pi_{0,s} \boxtimes \pi_{\lb, \star},
\]
so
\[
\sigma^H_\lb = \1 \boxtimes \pi_{\lb, 1}  + \1 \boxtimes \pi_{\lb, s} = I \boxtimes (\pi_{\lb, 1}  + \pi_{\lb, s}) - \pi^H_{0 + \lb,1} - \pi^H_{0 + \lb, s},
\]
implying
\[
\tr_{\sigma^H_\lb}(\eta^H_\lb) = -1.
\]
Finally
\begin{multline*}
\1 \boxtimes \1 = ( I - \pi_{0, 1} - \pi_{0,s}) \boxtimes ( I - \pi_{0, 1} - \pi_{0,s}) \\
= I \boxtimes I - I \boxtimes (\pi_{0,1} + \pi_{0,s}) - (\pi_{0,1} + \pi_{0,s}) \boxtimes I +  \pi^H_{0 + 0, 1} + \pi^H_{0+ 0, s},
\end{multline*}
so 
\[
\tr_{\1}(\eta^H_\lb) = 1.
\]
Since $\eta_\lb$ is supported on $H^0_\infty$, we similarly have
\[
\tr_{\chi}(\eta^H_\lb) = 1.
\]

In total, our $H$-term becomes a count
\begin{equation}\label{HRrep}
I^H(\eta^H_\lb \otimes \1_{K_H}) = \sum_{\pi \in \mc{AR}_{\disc, \ur}(H)} m_\disc^H(\pi) w^H(\pi_\infty),
\end{equation}
where $w^H$ is a weight
\[
w^H(\pi_\infty) = \begin{cases}
0 & \pi_\infty \text{ not cohomological of weight } \lb \\
1/2 & \pi_\infty \text{ one of the } \pi^H_{\lb, *} \\
-1 & \pi_\infty = \sigma^H_\lb \\
1 & \pi_\infty \text{ trivial  or $\chi$ and } \lb = 0
\end{cases}.
\]
Call the cohomological cases type I, II, and III in order.

\subsection{Reduction to Modular Form Counts}\label{HtoGL2}
We now recall two results from \cite{CR15}. Consider central isogeny $G \to G'$ of semisimple algebraic groups over $\Z$. If $\pi' = \pi'_\infty \otimes \pi'^{\infty}$ is an unramified, discrete automorphic representation of $G'$, let $R(\pi')$ be the set of unitary, admissible representations $\pi = \pi_\infty \otimes \pi^\infty$ of $G(\A)$ that satisfy:
\begin{itemize}
\item $\pi^\infty$ is unramified with set of Satake parameters $c^\infty(\pi^\infty)$ induced from that of $\pi'^\infty$ through $T_{G'}^G : \wh G' \to \wh G$. 
\item $\pi_\infty$ is a constituent of the restriction of $\pi'_\infty$ through $G(\R) \to G'(\R)$. 
\end{itemize}
Note that the size of $R(\pi')$ is the number of constituents of the restriction $\pi'_\infty|_{G(\R)}$. 
\begin{thm}[{\cite[prop. 4.7]{CR15}}]\label{CRthm}
Let $\pi$ be an automorphic representation of $G$.  Then
\[
m_\disc^G(\pi) = \sum_{\pi' : \pi \in R(\pi')} m_\disc^{G'}(\pi') [\pi_\infty, \pi'_\infty],
\]
where $[\pi_\infty, \pi'_\infty]$ is the multiplicity of $\pi_\infty$ in $\pi'_\infty|_{G(\R)}$. 
\end{thm}
We will apply this with $G = H$ and $G' = \PGL_2 \times \PGL_2$. Make similar definitions of type I, II, and III for representations of $[\PGL_2 \times \PGL_2](\R)$. Type I on $\PGL_2 \times \PGL_2$ restricts to the sum over a discrete $L$-packet on $H_\infty$. Type II and III on $\PGL_2 \times \PGL_2$ have irreducible restrictions. These restrictions partition the cohomological representations of $H$ except for $\chi$ so
\[
m_\disc^H(\pi_\infty \otimes \pi^\infty) = \sum_{c^\infty({\pi'}^\infty) \in (T_{G'}^H)^{-1}(c^\infty(\pi^\infty))} m_\disc^{G'}(\pi'_\infty \oplus {\pi'}^\infty)
\]
when $\pi_\infty \subseteq \pi'_\infty|_{H_\infty}$ and the multiplicity is $0$ when $\pi_\infty = \chi$. Now, we sum over the constituents of $\pi'_\infty|_{H_\infty}$ and the possible values of $c^\infty(\pi^\infty)$, noting that $T_{G'}^H$ is surjective. This gives:

\begin{cor}
\[
\sum_{\pi \in \mc{AR}_{\disc, \ur}(H)} m_\disc^H(\pi) w^H(\pi_\infty) = \sum_{\pi \in \mc{AR}_{\disc, \ur}(G')} m_\disc^{G'}(\pi) w^{G'}(\pi_\infty),
\]
where $w^{G'}$ is the weight
\[
w^{\PGL_2 \times \PGL_2}(\pi_\infty) = \begin{cases}
1 & \pi_\infty \text{ type I}\\
-1 & \pi_\infty \text{ type II}\\
1 & \pi_\infty \text{ type III}
\end{cases}
\]
that only differs from $w^H$ by multiplying the type I case by two. 
\end{cor}

%

Let $\mc S_k(1)$ be the set of normalized, level-$1$, weight-$k$ cuspidal (new)eigenforms. If $\lb = a\eps_1 + b \eps_2$, then type I representations on $\PGL_2 \times \PGL_2$ correspond to pairs in $S_{a+2}(1) \times S_{b+2}(1)$. Type II is a single form times the trivial representation and Type I is only the trivial representation.

\subsection{Final Formula for $S^H$}
Therefore, if
\[
S_k = |\mc S_k(1)|,
\]
we get:
\begin{equation}\label{SHformula}
I^H(\eta^H_{a \eps_1 + b \eps_2} \otimes \1_{K_H}) = (S_{a+2} - \1_{a = 0})(S_{b+2} - \1_{b = 0}),
\end{equation}
using canonical measure at finite places. By a classical formula (\cite[Thm. 3.5.2]{DS05} for example),
\begin{align*}
S_{a+2} &= \begin{cases} 0 & a+2 = 2 \text{ or } a+2 \text{ odd}\\
 \lfloor \f{a+2}{12} \rfloor - 1 & a+2 \equiv 2 \pmod{12} \\
\lfloor \f{a+2}{12} \rfloor & \text{else} \\
\end{cases}
.
\end{align*}

\section{A Jacquet-Langlands-style result}\label{JL} 

\subsection{First Form}
Generalizing \eqref{stabilization} slightly and substituting in \eqref{pstransferG2} and \eqref{eptransferG2} gives:
\begin{multline}\label{JLtf}
I^{G_2}(\varphi_{\pi_k} \otimes f^\infty) = I^{G_2^c}(\eta^{G_2^c}_{(k-2)\beta} \otimes f^\infty)  - I^H (\eta^H_{(3k-3)\eps_1+ (k-1)\eps_2} \otimes (f^\infty)^H) \\
+ I^H (\eta^H_{(3k-2)\eps_1 + (k-2)\eps_2} \otimes (f^\infty)^H).
\end{multline}
for any unramified function $f^\infty$ (we use here that $(G_2^c)^\infty =(G_2)^\infty$). This will let us describe the set $\mc Q_k(1)$ for $k > 2$ in terms of certain representations of $G_2^c$ and $H$.

Choose $\pi = \pi_k \otimes \pi^\infty \in \mc Q_k(1)$. Since $\pi^\infty$ is unramified, it can be described by a sequence of Satake parameters: for each prime $p$, a semisimple conjugacy class $c_p(\pi^\infty) \in [\wh {G_2}]_\ssm$ (note that $G_2$ is split so we don't need to worry about the full Langlands dual and see \cite[\S3.2]{ST16} for full background).

The endoscopic datum for $H$ also gives an embedding $\wh H \into \wh{G_2}$ (noting again that everything is split) whose image contains a chosen maximal torus and therefore induces a map 
\[
T^{G_2}_H : [\wh H]_{\ssm} \onto [\wh {G_2}]_{\ssm}.
\]
The fibers of this map are $\Om_{G_2}$-orbits of conjugacy classes in $H$ and have size $3$ at $G_2$-regular elements. 

\begin{prop}
Let $k > 2$ and $\pi^\infty$ an unramified representation of $(G_2)^\infty$. Then
\begin{multline*}
m^{G_2}_\disc(\pi_k \otimes \pi^\infty) = m^{G_2^c}_\disc(V_{(k-2)\beta} \otimes \pi^\infty) - \f12 |S^H(\pi^\infty, (3k-3)\eps_1+ (k-1)\eps_2)|  \\
+ \f12 |S^H(\pi^\infty, (3k-2)\eps_1+ (k-2)\eps_2)|.
\end{multline*}
Recall here that $V_\lb$ is the finite dimensional representation of $G_2^c$ with highest weight $\lb$. Also, $S^H(\pi^\infty, \lb)$ is the multiset of $\pi_\infty \otimes \pi_1^\infty \in \mc{AR}_\disc(H)$ with multiplicity such that:
\begin{itemize}
\item $\pi_\infty \in \Pi^H_\disc(\lb)$,
\item For all $p$, $c_p(\pi_1^\infty) \in (T^{G_2}_H)^{-1}(c_p(\pi^\infty))$. 
\end{itemize}
\end{prop}

\begin{proof}
This is a standard Jacquet-Langlands-style argument. Through the Satake isomorphism, each $f_p$ can be thought of as a function $[\wh{G_2}]_{\ssm} \to \C$ through $f_p(c_p(\pi)) = \tr_{\pi_p}(f_p)$. It is in fact a Weyl-invariant regular function on a maximal torus in $\wh{G_2}$. The full version of the fundamental lemma (see the introduction to \cite{Hal95} for example) shows that
\[
f_p^H(c_p) = f_p(T^{G_2}_H(c_p))
\]
for all $c_p \in \wh H$. 

There are only finitely many sequences $c_p(\pi_1^\infty)$ and $T^{G_2}_H(c_p(\pi_1^\infty))$ for $\pi_1^\infty$ the unramified finite component of an automorphic representation either:
\begin{itemize}
\item of $G_2$ with infinite part $\pi_k$,
\item of $G_2^c$ with infinite part $V_{(k-2)\beta}$,
\item or of $H$ with infinite part in $\Pi_\disc((3k-3)\eps_1+ (k-1)\eps_2)$ or $\Pi_\disc((3k-2)\eps_1+ (k-2)\eps_2)$. 
\end{itemize} 
Therefore we can choose an $f^\infty$ that is $0$ on all of these sequences $c_p(\pi_1^\infty)$ except $1$ on exactly the sequence $c_p(\pi^\infty)$ (this reduces to finding Weyl-invariant polynomials on $(\C^\times)^2$ that take specified values on certain Weyl orbits). The result follows from plugging this $f^\infty$ into equation \eqref{JLtf}, noting that the $w^H$ from equation \eqref{HRrep} is always $1/2$ in the relevant cases.
\end{proof}

\subsection{In terms of Modular Forms}
We can use the argument from section \ref{HtoGL2} to reduce the $H$-multiplicity terms to $\PGL_2$-multiplicity ones.

First, we have a map on conjugacy classes 
\[
T^{ H}_{\PGL_2 \times \PGL_2} : [\wh {\PGL_2 \times \PGL_2}]_{\ssm} \onto [\wh H]_{\ssm}.
\]
Since the first group is $\SL_2 \times \SL_2(\C)$, the fibers of this map are of the form $\{c, -c\}$ for some $c \in [\SL_2 \times \SL_2(\C)]_\ssm$. Composing then gives map
\[
T^{G_2}_{\PGL_2 \times \PGL_2} : [\wh {\PGL_2 \times \PGL_2}]_{\ssm} \onto [\wh{G_2}]_{\ssm}.
\]
This allows us to define $S^{\PGL_2 \times \PGL_2}(\pi^\infty, \lb)$ analogous to $S^H(\pi^\infty, \lb)$ for all $\lb = a \eps_1 + b \eps_2$ with both $a$ and $b$ even. For indexing purposes, set it to be empty when $a$ and $b$ aren't even. 

Formula \eqref{SHformula} also gives us that $S^H(\pi^\infty, a \eps_1 + b \eps_2) = \emptyset$ when $a$ and $b$ aren't both even. Recall from \S\ref{HtoGL2} that the restriction of discrete series $\pi^{\PGL_2 \times \PGL_2}_\lb$ to $H(\R)$ has as components the two members of the $L$-packet $\Pi_\disc^H(\lb)$. Therefore a similar analysis using theorem \ref{CRthm} shows that 
\[
|S^H(\pi^\infty, \lb)| = 2|S^{\PGL_2 \times \PGL_2}(\pi^\infty, \lb)|.
\]

Finally, $\PGL_2$ is a quotient of $\GL_2$ by a central torus with trivial Galois cohomology, so automorphic representations on $\PGL_2$ are just those on $\GL_2$ with all components having trivial central character. Recalling injection
\[
\iota : [\SL_2 \times \SL_2(\C)]_\ssm \into [\GL_2 \times \GL_2(\C)]_\ssm,
\]
this gives:
\begin{cor}\label{JLG2}
Let $k > 2$ and $\pi^\infty$ an unramified representation of $(G_2)^\infty$. Then
\begin{multline*}
m^{G_2}_\disc(\pi_k \otimes \pi^\infty) = m^{G_2^c}_\disc(V_{(k-2)\beta} \otimes \pi^\infty) - |S^{\GL_2 \times \GL_2}(\pi^\infty, (3k-3)\eps_1+ (k-1)\eps_2)|  \\
+  |S^{\GL_2 \times \GL_2}(\pi^\infty, (3k-2)\eps_1+ (k-2)\eps_2)|.
\end{multline*}
Recall here that $V_\lb$ is the finite dimensional representation of $G_2^c$ with highest weight $\lb$. Also, $S^{\GL_2 \times \GL_2}(\pi^\infty, \lb)$ is the set of $\pi_\infty \otimes \pi_1^\infty \in \mc{AR}_\disc(\GL_2 \times \GL_2)$ such that:
\begin{itemize}
\item $\pi_\infty$ is the discrete series $\pi^{\GL_2 \times \GL_2}_\lb$,
\item For all $p$, $c_p(\pi_1^\infty) = \iota(c'_p)$ for some $c'_p \in (T^{G_2}_{\PGL_2 \times \PGL_2})^{-1}(c_p(\pi^\infty))$. Here $\iota$ is the map $ [\SL_2 \times \SL_2(\C)]_\ssm \into [\GL_2 \times \GL_2(\C)]_\ssm$. 
\end{itemize}
\end{cor}

Of course, since all infinite factors in sight are discrete series, we may again replace the $m_\disc$ by $m_\cusp$ using \cite{Wal84}. 

Note of course that $S^{\GL_2 \times \GL_2}(\pi^\infty, a\eps_1 + b \eps_2) = \emptyset$ unless both $a$ and $b$ are even. Therefore, we can interpret this as, for $k > 2$:
\begin{itemize}
\item
If $k$ is even: $\mc Q_k(1)$ is the corresponding set of representations transferred from $G_2^c$ \emph{in addition to} representations transferred from pairs of cuspidal eigenforms in $\mc S_{3k}(1) \times \mc S_{k}(1)$.  
\item
If $k$ is odd: $\mc Q_k(1)$ is the corresponding set of representations transferred from $G_2^c$ \emph{except for} representations that are also transferred from pairs of cuspidal eigenforms in $\mc S_{3k-1}(1) \times \mc S_{k+1}(1)$.
\end{itemize}

Results for level $>1$ would be a lot more complicated since formula \eqref{stabilization} would have many further hyperendoscopic terms and the comparison to $\GL_2 \times \GL_2$ would not work as nicely.

\section{Counts of forms}\label{finalcount}
\subsection{Formula in terms of $I^{G_2^c}$}
To get counts instead of a list, combining formulas \eqref{ssideG2},\eqref{JLtf}, and \eqref{SHformula} gives that 
\begin{multline}\label{finalformula}
|\mc Q_k(1)| = I^{G_2^c}(\eta^{G_2^c}_\lb \otimes \1_{K^\infty_{G_2^c}}) - (S_{3k-1} -\1_{3k-3 = 0})(S_{k+1} -\1_{k-1 = 0}) \\+ (S_{3k} -\1_{3k-2 = 0})(S_{k} -\1_{k-2 = 0}),
\end{multline}
where $S_k$ as before represents the count of classical modular forms of weight $k$. 

Substituting in the formulas for $S_k$, for $k > 2$:
\begin{multline*}
|\mc Q_k(1)| = I^{G_2^c}(\eta_\lb \otimes \1_{K^\infty_{G_2^c}}) \\+
\begin{cases}
\lfloor \frac k4 \rfloor \lf( \lfloor \frac k{12} \rfloor - 1 \ri) & k \equiv 2 \pmod{12} \\
\lfloor \frac k4 \rfloor  \lfloor \frac k{12} \rfloor   & k \equiv 0,4,6,8,10 \pmod{12} \\
-\lf(\lfloor \frac {3k-1}{12}  \rfloor - 1\ri) \lf( \lfloor \frac {k+1}{12} \rfloor - 1 \ri) & k \equiv 1 \pmod{12} \\
-\lf(\lfloor \frac {3k-1}{12}  \rfloor - 1\ri) \lfloor \frac {k+1}{12} \rfloor  & k \equiv 5,9 \pmod{12} \\
-\lfloor \frac {3k-1}{12} \rfloor \lfloor \frac {k+1}{12} \rfloor & k \equiv 3,7,11 \pmod{12} \\
\end{cases}.
\end{multline*}

\subsection{Computing $I^{G_2^c}$}
The group $G_2^c(\R)$ is compact so the $I^{G_2^c}$ term takes a very simple form: $L^2(G_2^c(\Q) \bs G_2^c(\A))$ decomposes as a direct sum of automorphic representations and the EP-functions $\eta_\lb$ are just scaled matrix coefficients of the finite-dimensional representations $V_\lb$ with highest weight $\lb$ on $G_2^c(\R)$. Therefore
\[
I^{G_2^c}(\eta_\lb \otimes \1_{K^\infty_{G_2^c}}) = \sum_{\pi \in \mc{AR}(G_2^c)} \1_{\pi_\infty = V_\lb} \tr_{\pi^\infty}(\1_{K^\infty_{G_2^c}}),
\]
which is just counting the number of unramifed automorphic reps of $G_2^c$ that have infinite component $V_\lb$. 

For reader convenience, we now explain in detail an argument well known to experts. Since unramified representations have $1$-dimensional spaces of $K^\infty$-fixed vectors, taking $K^\infty_{G_2^c}$ invariants sends each such $\pi$ to a linearly independent copy of $V_\lb$ that together span the $V_\lb$-isotypic component of 
\[
L^2(G_2^c(\Q) \bs G_2^c(\A) / K^\infty_{G_2^c}) = L^2(G_2^c(\Z) \bs G_2^c(\R)) \subseteq L^2(G_2^c(\R)). 
\]
By Peter-Weyl, $L^2(G_2^c(\R))$ has $V_\lb$-isotypic component $V_\lb^{\oplus \dim V_\lb}$. In fact, this component for both the left- and right-actions is the same subspace. Therefore the number of copies of $V_\lb \subseteq L^2(G_2^c(\Z) \bs G_2^c(\R))$ is $\dim \lf(V_\lb^{G_2^c(\Z)}\ri)$ by a dimension count.

Summarizing:
\begin{equation}\label{IG2c}
I^{G_2^c}(\eta_\lb \otimes \1_{K^\infty_{G_2^c}}) = \dim \lf(V_\lb^{G_2^c(\Z)}\ri). 
\end{equation}
A PARI/GP 2.5.0 program in the online appendix to \cite{CR15} computes this for all $\lb$ by pairing the trace character of $V_\lb|_{G_2(\Z)}$ with the trivial character.

An explicit paper formula for this computation is more-or-less written out in an honors thesis of Steven Sullivan \cite{Sul13}. Sullivan writes out the traces of all $16$ conjugacy classes in $G_2^c(\Z)$ against $V_{(k-2) \beta}$ as polynomials of $k$ with coefficients that are sums of $k$th powers of $7$th, $8$th, and $12$th roots of unity. This gets a polynomial expression for the trace character pairing and therefore $I^{G_2^c}(\eta_\lb \otimes \1_{K^\infty_{G_2^c}})$ in cases $\pmod{168}$.  Simplifications in Mathematica give a reasonable closed-form version in section \ref{explicit}. 

It is important to note here that getting the explicit descriptions and sizes of the conjugacy classes in $G_2^c(\Z)$ was non-trivial and required some trickery in both Sullivan's and Chenevier-Ta\"ibi's computations. This step would be an obstacle to any generalizations.

\subsection{Table of Counts}
Table \ref{tabledims} gives values of $|\mc Q_k(1)|$ for $k=3$ to $52$ produced by formula \eqref{finalformula} and \cite{CR15}'s table for formula \eqref{IG2c}. The lowest-weight example is bolded, although this work does not rule out the existence of an example with weight $2$ or weight $1$ (as defined by \cite[\S1.1]{Pol20}). 

\begin{table}[ht]
\centering
\caption{Counts of discrete, quaternionic automorphic representations of level $1$ on $G_2$.}
\begin{tabular}{cc||cc||cc||cc||cc} \toprule
$k$ & $|\mc Q_k(1)|$ & $k$ & $|\mc Q_k(1)|$ & $k$ & $|\mc Q_k(1)|$ & $k$ & $|\mc Q_k(1)|$ & $k$ & $|\mc Q_k(1)|$ \\ \midrule
 3 & 0 & 13 & 5 & 23 & 76 & 33 & 478 & 43 & 1792 \\
 4 & 0 & 14 & 13 & 24 & 126 & 34 & 610 & 44 & 2112 \\
 5 & 0 & 15 & 8 & 25 & 121 & 35 & 637 & 45 & 2250 \\
 \textbf 6 & \textbf 1 & 16 & 23 & 26 & 175 & 36 & 807 & 46 & 2619 \\
 7 & 0 & 17 & 17 & 27 & 173 & 37 & 849 & 47 & 2790 \\
 8 & 2 & 18 & 37 & 28 & 248 & 38 & 1037 & 48 & 3233 \\
 9 & 1 & 19 & 30 & 29 & 250 & 39 & 1097 & 49 & 3447 \\
 10 & 4 & 20 & 56 & 30 & 341 & 40 & 1332 & 50 & 3938 \\
 11 & 1 & 21 & 50 & 31 & 349 & 41 & 1412 & 51 & 4201 \\
 12 & 9 & 22 & 83 & 32 & 460 & 42 & 1686 & 52 & 4780 \\ \bottomrule
\end{tabular}
\label{tabledims}
\end{table}.

\subsection{Explicit Formula}\label{explicit}

For the reader's amusement, we build off the work of \cite{Sul13} to present a closed-form formula for $|\mc Q_k(1)|$ that fits in a few lines:
\begingroup
\addtolength{\jot}{1em}
\begin{align*}
&|\mc Q_{n+2}(1)| = \\
& \f1{12096} \f1{120}(n+1)(3n+4)(n+2)(3n+5)(2n+3) + \f1{216}\f16 (n+1) (n+ 2 ) (2 n+3) \\ 
&+ \f5{192}\f18 
\begin{cases}
(n+2) ( 3 n+4) & n = 0 \pmod 2 \\
-(n+1) (3 n+5)& n = 1 \pmod 2 \\
\end{cases} 
+ \f1{18}\begin{cases}
\f{2n}3 + 1& n = 0 \pmod 3 \\
 - \lfloor \f n3  \rfloor - 1 & n = 1,2 \pmod 3
\end{cases} \\
&+ \f1{32}\begin{cases}
\f{3n}2 + 10& n = 0 \pmod 4 \\
 6\lfloor \f{n}4  \rfloor - 4 & n = 1 \pmod 4 \\
  - 2\lfloor \f{n}4  \rfloor - 2 & n = 2,3 \pmod 4
\end{cases} 
 + \f1{24} \begin{cases}
3\lfloor \f{n}6 \rfloor+ 5& n = 0,1 \pmod 6 \\
3 \lfloor \f{n}6  \rfloor -2 & n = 2,3 \pmod 6 \\
3  \lfloor \f{n}6  \rfloor +3 & n = 4,5 \pmod 6
\end{cases} \\
&+ \f17 \begin{cases}
1 & n = 0 \pmod 7 \\
-1 & n = 4 \pmod 7 \\
0 & n = 1,2,3,5,6 \pmod 7  
\end{cases}  
+ \f14\begin{cases}
1 & n = 0 \pmod 8 \\
-1 & n = 5 \pmod 8 \\
0 & n = 1,2,3,4,6,7 \pmod 8
\end{cases} \\
&+ \begin{cases}
\lfloor \frac {n+2}4 \rfloor \lf( \lfloor \frac {n+2}{12} \rfloor - 1 \ri) & n = 0 \pmod{12} \\
\lfloor \frac {n+2}4 \rfloor  \lfloor \frac {n+2}{12} \rfloor   & n = 2,4,6,8,10 \pmod{12} \\
-\lf(\lfloor \frac {3n+5}{12}  \rfloor - 1\ri) \lf( \lfloor \frac {n+3}{12} \rfloor - 1 \ri) & n = 11 \pmod{12} \\
-\lf(\lfloor \frac {3n+5}{12}  \rfloor - 1\ri) \lfloor \frac {n+3}{12} \rfloor  & n = 3,7 \pmod{12} \\
-\lfloor \frac {3n + 5}{12} \rfloor \lfloor \frac {n+3}{12} \rfloor & n = 1,5,9 \pmod{12} \\
\end{cases}.
\end{align*}
\endgroup

\bibliographystyle{amsalpha}
\bibliography{/Users/Rahul/Documents/Rahul/Tbib}

\end{document}